\documentclass{amsart}
\usepackage{amssymb,amsmath}
 \usepackage{hyperref}

\newtheorem{theorem}{Theorem}[section]
\newtheorem{corollary}[theorem]{Corollary}

\newtheorem{lemma}[theorem]{Lemma}
\newtheorem{proposition}[theorem]{Proposition}

\theoremstyle{definition}
\newtheorem{definition}[theorem]{Definition}

\newcommand{\Z}{{\mathbb Z}}
\newcommand{\R}{{\mathbb R}}
\newcommand{\C}{{\mathbb C}}

\newcommand{\D}{{\mathbb D}}

\newcommand{\X}{{\mathbb S}}

\renewcommand{\S}{{\mathbb S}}
\newcommand{\PP}{{\mathbb P}}
\newcommand{\E}{{\mathcal E}}

\newcommand{\GL}{{\mathrm{GL}}}
\newcommand{\SL}{{\mathrm{SL}}}

\newcommand{\UH}{{\mathcal{UH}}}

\title[Uniform Hyperbolicity and Spectra of CMV Matrices]{Characterizations of Uniform Hyperbolicity and Spectra of CMV Matrices}

\author[D.\ Damanik]{David Damanik}

\author[J.\ Fillman]{Jake Fillman}

\author[M.\ Lukic]{Milivoje Lukic}

\author[W.\ Yessen]{William Yessen}

\email{damanik@rice.edu}
\email{fillman@vt.edu}
\email{mlukic@math.toronto.edu}
\email{yessen@rice.edu}

\thanks{D.\ D.\ and J.\ F.\ were supported in part by NSF grants DMS--1067988 and DMS--1361625.}
\thanks{M.\ L.\ was supported in part by NSF grant DMS--1301582.}
\thanks{W.\ Y.\ was supported by the NSF Mathematical Sciences Postdoctoral Research Fellowship DMS-1304287}

\begin{document}
\maketitle

\begin{abstract}
We provide an elementary proof of the equivalence of various notions of uniform hyperbolicity for a class of $\GL(2,\C)$ cocycles and establish a Johnson-type theorem for extended CMV matrices, relating the spectrum to the set of points on the unit circle for which the associated Szeg\H{o} cocycle is not uniformly hyperbolic.
\end{abstract}

\section{Introduction}

This paper discusses an interface between dynamical systems and spectral theory. Specifically, there is an intimate relation between hyperbolicity of $\GL(2,\C)$ cocycles and the spectra of certain finite-difference operators on $\ell^2(\Z)$. This connection goes back to the landmark paper of Johnson \cite{J}, in which he characterizes the spectra of a class of self-adjoint linear differential operators as the set of energies for which the associated differential equation fails to admit an exponential dichotomy.

The present paper originally came about as the authors were writing \cite{DFLY}, which relies heavily on this connection to prove that the one-dimensional ferromagnetic Ising model does not exhibit a phase transition in the thermodynamic limit. In order to understand this particular interface between dynamics and spectra, we needed to make use of several different equivalent notions of uniform hyperbolicity. The equivalence of these notions is well-known to experts in the field, but we were unable to locate a satisfactorily complete proof of this fact in the literature, so we decided to write an appendix to \cite{DFLY} to accomplish precisely that goal. As we began to write, we realized that one could recover the aforementioned characterization of the spectra of CMV matrices in a very simple manner by way of generalized eigenfunctions, which led to a second appendix elucidating this point of view. In particular, even though this connection between resolvent sets and uniform hyperbolicity for CMV matrices is well-known (it is essentially \cite[Theorem~5.1]{GJ96}), we feel that the proof contained herein is a worthwhile addition to the literature as it is more constructive and elementary than Geronimo-Johnson's proof. Ultimately, the appendices were too long to be included in the published version of \cite{DFLY}, but we still feel that detailed proofs of these theorems should have a place in the literature.

Let us now state the results precisely. To fix notation, let $\C\PP^1$ denote the complex projective line, the set of one-dimensional complex subspaces of $\C^2$. The projective line is equipped with a natural angle metric defined by
$$
d(V,W)
=
\angle(V,W)
=
\arccos \big( | \langle v, w \rangle| \big),
\quad
v \in V, \, w \in W, \, \|v \| = \|w\| = 1,
$$
where $\langle \cdot, \cdot \rangle$ denotes the standard Hermitian inner product on $\C^2$.  This endows $\C\PP^1$ with a topology homeomorphic to that of the sphere $\S^2$.

Let $G \subseteq \GL(2,\C)$ denote the group of matrices whose determinant has absolute value one.  In particular, $ G $ contains the groups $\SL(2,\R)$, $\mathrm O(2)$, $\mathrm U(2)$, $\mathrm U(1,1)$, among others. If $M \in \GL(2,\C)$, then the action of $M$ on $\C^2$  descends to a  map on $\C\PP^1$ which we will denote by the same letter.

Now, fix a compact metric space $\Omega$, a homeomorphism $T: \Omega \to \Omega$, and a continuous map $A: \Omega \to G$. We define the associated cocycle $(T,A):\Omega \times \C^2 \to \Omega \times \C^2$ by
$$
(T,A) : (\omega,v) \mapsto (T\omega, A(\omega)v),
\quad
\omega \in \Omega, \, v \in \C^2.
$$
For $n \in \Z$, define $A^n : \Omega \to G$ by
$$
A^n(\omega)
=
\begin{cases}
A(T^{n-1}\omega) \cdots A(\omega) & n \geq 1 \\
I & n = 0 \\
A(T^n\omega)^{-1} \cdots A(T^{-1}\omega)^{-1} & n \leq -1
\end{cases},
$$
so that iterates of the skew product obey $(T,A)^n = (T^n,A^n)$. The following definition collects the various notions which we will relate to hyperbolicity of $(T,A)$.

\begin{definition}
We say that a cocycle $(T,A)$ exhibits \emph{uniform exponential growth} if there are constants, $C>0$, $\lambda>1$ with the property that
\begin{equation} \label{e.exp.growth}
\| A^n(\omega) \|
\geq
C \lambda^{|n|}
\end{equation}
for all $n \in \Z$ and $\omega \in \Omega $.

We say that $(T,A)$ admits an \emph{invariant exponential splitting} if there exist constants $c > 0, L > 1 $ and continuous maps $ \Lambda^s,\Lambda^u : \Omega \to \C \PP^1 $ such that the following statements hold.
\begin{itemize}
\item[{\rm (a)}] (Invariance) For all $\omega \in \Omega$, one has
\begin{equation} \label{eq:mnflds.inv}
A(\omega) \Lambda^s (\omega) = \Lambda^s(T\omega), \quad A(\omega)\Lambda^u (\omega) = \Lambda^u(T\omega).
\end{equation}

\item[{\rm (b) }] (Contraction) For all $n \in \Z_+$, $\omega\in\Omega$, $v_s \in \Lambda^s(\omega)$, and $ v_u \in \Lambda^u(\omega) $, one has
\begin{equation} \label{e.exp.dichotomy}
\| A^n   (\omega) v_s \| \leq c L^{-n} \|v_s\|,
\quad
\| A^{-n}(\omega) v_u \| \leq c L^{-n} \|v_u\|.
\end{equation}
\end{itemize}
One can paraphrase \eqref{e.exp.dichotomy} by saying that vectors in $\Lambda^s$ decay exponentially in forward time, while vectors in $\Lambda^u$ decay exponentially in reverse time.

Lastly, we say that $(T,A)$ enjoys a \emph{Sacker-Sell solution} if there exist $\omega \in \Omega$ and $v \in \S^3 := \left\{ v \in \C^2 : \|v\|=1 \right\} $  so that
\[
\|A^n(\omega) v \| \leq 1
\]
for every $ n \in \Z$  (compare \cite{sacksell1}, \cite{sacksell2}, and \cite{selgrade}).
\end{definition}

\begin{theorem} \label{t.uh.splitting}
Let $(T,A)$ be defined as above.  Then the following are equivalent.

\begin{itemize}
\item[\rm (a)] $(T,A)$ exhibits uniform exponential growth.

\item[\rm (b)] $(T,A)$ admits an invariant exponential splitting.

\item[\rm (c)] $(T,A)$ does not enjoy a Sacker-Sell solution.
\end{itemize}

Whenever {\rm (a)}, {\rm (b)}, and {\rm (c)} hold, there is a uniform constant $ \gamma > 0 $ so that $ d(\Lambda^s(\omega), \Lambda^u(\omega)) \geq \gamma $ for every $\omega$.

\end{theorem}

\begin{definition}
Whenever one (and hence all) of the conditions above hold, we say that the cocycle $(T,A)$ is \emph{uniformly hyperbolic}, which we sometimes denote by $(T,A) \in \UH $.
\end{definition}

Let us briefly comment on the literature. The main elements of Theorem~\ref{t.uh.splitting} can be found in, for example, \cite{BG}, \cite{Y04}, and \cite{Z} (some similar ideas can also be found in \cite[Theorem 8.1]{LS99}, though the authors in this case are mainly interested in subordinacy).  However, we feel that our presentation does fill a gap left by other work. In particular, the canonical reference for Theorem \ref{t.uh.splitting}, \cite[Proposition 2]{Y04}, does not prove the exponential decay estimates in the implication (a)$\implies$(b). Additionally, although \cite[Theorem A]{BG} is much more general than the equivalence (a)$\iff$(b) in our statement of Theorem~\ref{t.uh.splitting}, their proof uses some potent weaponry from ergodic theory, while our proof is entirely constructive and deterministic. Large portions of our exposition in Section~\ref{s:uh} are inspired by Zhang's preprint -- however, \cite{Z} appeals to an invariant cone criterion, while our argument is more self-contained.

We may use the characterizations of hyperbolicity to study the spectra of two-sided CMV matrices, which are unitary operators on $\ell^2(\Z)$ constructed as follows. Take a sequence $(\alpha_n)_{n\in\Z}$, where $\alpha_n \in \D := \{z \in \C : |z| < 1 \}$ for each $n \in \Z$. The corresponding \emph{extended CMV matrix} enjoys the following matrix representation with respect to the standard basis of $\ell^2(\Z)$:
$$
\E
=
\begin{pmatrix}
\ddots & \ddots & \ddots &&&&&  \\
 & -\overline{\alpha_0}\alpha_{-1} & \overline{\alpha_1}\rho_0 & \rho_1\rho_0 &&& & \\
& -\rho_0\alpha_{-1} & -\overline{\alpha_1}\alpha_0 & -\rho_1\alpha_0 &&& & \\
&  & \overline{\alpha_2}\rho_1 & -\overline{\alpha_2}\alpha_1 & \overline{\alpha_3}\rho_2 & \rho_3\rho_2 & & \\
& & \rho_2\rho_1 & -\rho_2\alpha_1 & -\overline{\alpha_3}\alpha_2 & -\rho_3\alpha_2 &  &  \\
&&&& \overline{\alpha_4}\rho_3 & -\overline{\alpha_4}\alpha_3 & \overline{\alpha_5}\rho_4 &\\
&&&& \rho_4\rho_3 & -\rho_4\alpha_3 & -\overline{\alpha_5}\alpha_4 &   \\
& &&&& \ddots & \ddots &  \ddots
\end{pmatrix},
$$
where $\rho_n = \left( 1 - |\alpha_n|^2 \right)^{1/2}$, and the terms of the form $-\overline{\alpha_n}\alpha_{n-1}$ comprise the main diagonal, i.e., $\langle \delta_n, \E \delta_n \rangle = -\overline{\alpha_n}\alpha_{n-1}$ for $n \in \Z$. A particularly interesting situation is the case in which the $\alpha$'s are obtained by sampling along the orbits of some topological dynamical system. Specifically, let $\Omega$ and $T$ be as before, and suppose $f:\Omega \to \D$ is continuous. For each $\omega \in \Omega$, one obtains a corresponding extended CMV matrix $\E_\omega$ defined by $\alpha_\omega(n) = f(T^n \omega)$. The study of $\E_\omega$ is intimately related to the Szeg\H{o} cocycle, defined by $A_z(\omega) = S(f(\omega),z)$, where
$$
S(\alpha,z)
=
\frac{1}{\rho}
\begin{pmatrix}
z & - \overline{\alpha} \\
-\alpha z & 1
\end{pmatrix},
\quad
\rho
=
\left( 1 -|\alpha|^2 \right)^{1/2},
\quad
\text{for } \alpha \in \D, \, z \in \partial \D.
$$
Specifically, we have the following theorem.
\begin{theorem}\label{t:minspec:johnson}
If $(\Omega,T)$ is minimal, then there is a uniform compact set $\Sigma \subseteq \partial \D$ with $\sigma(\E_\omega) = \Sigma$ for every $\omega \in \Omega$.  Moreover, this uniform spectrum is characterized as $\Sigma = \partial \D \setminus U$, where
$$
U
=
\{ z \in \partial \D : (T,A_z) \in \UH \}.
$$
\end{theorem}

For Schr\"odinger operators, certain spectral characterizations are more naturally made in terms of generalized eigenfunctions and others in terms of transfer matrices, but the two are very closely related since generalized eigenfunctions are generated by the transfer matrices. However, for CMV matrices, the recursion relation for generalized eigenfunctions is not given by the standard transfer matrices (generated by the Szeg\H o cocycle), but instead by transfer matrices generated by the Gesztesy--Zinchenko \cite{GZ06} cocycle. In our proof of Theorem 2, we note a simple relation between the two cocycles which ultimately allows us to relate uniform hyperbolicity of the Szeg\H o cocycle to exponential asymptotics of generalized eigenfunctions and prove Theorem~\ref{t:minspec:johnson}. This relation has other applications; in \cite{LO}, it will be used to conclude boundedness of generalized eigenfunctions from a Pr\"ufer variable approach adapted to the Szeg\H o cocycle. In \cite{DFO}, it will be used to relate the behavior of Szeg\H{o} transfer matrices to quantitative spreading estimates for quantum walks.

The structure of the paper is as follows. In Section~\ref{s:uh}, we prove Theorem~\ref{t.uh.splitting}. In Section~\ref{s:genef} we first present a CMV version of a standard result for Schr\"odinger operators which characterizes the spectrum as the closure of the set of generalized eigenvalues. The proofs are simple modifications of those in the Schr\"odinger case, but we opted to present them here for the convenience of the reader. It is worth pointing out that this discussion of generalized eigenfunctions is a special case of general, powerful results; see \cite{B}. Finally, we then use this characterization of the spectrum together with Theorem \ref{t.uh.splitting} to prove Theorem~\ref{t:minspec:johnson}.

\section{Uniform hyperbolicity} \label{s:uh}

Given a cocycle $(T,A)$, it is helpful to be able to relate behavior of $A^n$ on the right half-line to its behavior on the left half-line and vice versa. In that regard, the identity
\begin{equation} \label{eq:cocyc:rev}
A^{-n}(T^n \omega)
=
A^n(\omega)^{-1}
\end{equation}
turns out to be highly useful.  We also note that $ \|M^{-1}\| = \|M\| $ for $2\times 2$ matrices $M$ with $|\det(M)|=1$.

To prove Theorem~\ref{t.uh.splitting}, we will make use of the following simple variant of the singular value decomposition.

\begin{lemma} \label{l.sing.mnflds}
Given $A \in G$ such that $ \|A\|> 1 $, there exist complex lines $ S = S(A)$, $U = U(A) \in \C\PP^1 $ so that
\begin{equation} \label{eq:sing.mnflds.scaling}
\| A v_s \| = \|A\|^{-1} \|v_s\|, \quad
\| A v_u \| = \|A\|      \|v_u\|
\end{equation}
for all $v_s \in S, v_u \in U$.  Moreover, $S$ and $ U $ are orthogonal and satisfy
\begin{equation} \label{eq:sing.mnflds.mult}
A \cdot S(A) = U\left( A^{-1} \right), \quad
A \cdot U(A) = S\left( A^{-1} \right).
\end{equation}
Additionally, given $V \in \C\PP^1$, if $ \|A v\| = R \|v\| $ for $v \in V$, then  $\theta = \angle\left(V, S \right)$  satisfies
\begin{equation} \label{e.sing.mnfld.pert}
\sqrt{ \frac{R^2 - \|A\|^{-2}}{\|A\|^2 - \|A\|^{-2}}}
\leq
\theta
\leq
\frac{\pi}{2}\sqrt{ \frac{R^2 - \|A\|^{-2}}{\|A\|^2 - \|A\|^{-2}}}.
\end{equation}
A similar estimate holds for $U$.  In particular, the mappings $ S,U:G \setminus \mathrm{U}(2) \to \C\PP^1 $ are continuous.
\end{lemma}

\begin{proof}
Take $S$ to be the eigenspace of the Hermitian matrix $A^* A$ corresponding to the eigenvalue $\|A\|^{-2}$ and $U$ the eigenspace of $A^* A$ corresponding to the eigenvalue $ \|A\|^2$.  Verification of \eqref{eq:sing.mnflds.scaling}, \eqref{eq:sing.mnflds.mult}, and \eqref{e.sing.mnfld.pert} is then a pleasant exercise in linear algebra.
\end{proof}

\begin{proof}[Proof of Theorem \ref{t.uh.splitting}]

(a)$\implies $(b) We will frequently need to compare asymptotic behavior of sequences of nonnegative functions $f_n,g_n: \Omega \to \R_{\geq 0}$.  To that end, we will say $f_n \lesssim g_n$ if there exists a uniform constant $C_0>0$ so that $ f_n(\omega) \leq C_0 g_n(\omega) $ for every $\omega \in \Omega$ and every $n \in \Z_+$.  Similarly, if $x$ and $y$ are sequences, we write $ x_n \lesssim y_n $ if there exists $C_1 > 0$ with $ x_n \leq C_1 y_n $ for all $n \in \Z_+$.  It is quite easy for one to fill in the implicit constants, but they tend to clutter the exposition, so we omit them in the course of the proof.

The overall strategy is quite simple: the estimate \eqref{e.exp.growth} implies that the most contracted subspace of $A^n(\omega)$ converges exponentially fast to a well-defined limit as $n \to \infty$, which is a natural candidate for $\Lambda^s(\omega)$.  Proving the exponential decay estimates turns out to be slightly tricky -- the main idea is to show first that $\Lambda^s$ and $\Lambda^u$ are uniformly separated, and then use this to show that vectors in $\Lambda^s$ grow exponentially quickly in reverse time and then combine this with \eqref{eq:cocyc:rev} to show that vectors in $\Lambda^s$ decay exponentially in forward time.  The details follow.

Fix $C>0$ and $\lambda > 1$ so that \eqref{e.exp.growth} holds for all $n \in \Z$ and all $\omega \in \Omega$, and put
$$
B = \max_{\omega \in \Omega} \|A(\omega)\|,
$$
which is finite by continuity of $A$ and compactness of $\Omega$. If $n$ is sufficiently large, $\|A^n(\omega)\| > 1$; for such $n$, put $\Lambda_n^s(\omega) = S(A^n(\omega))$, choose unit vectors $v_n(\omega) \in \Lambda_n^s(\omega)$, and observe that
$$
\|A^n(\omega) v_{n+1}(\omega) \|
=
\|A(T^n \omega)^{-1} A^{n+1}(\omega) v_{n+1}(\omega) \|
\leq
B \|A^{n+1}(\omega)\|^{-1}.
$$
In particular, for $n > \log \left( 2C^{-1} \right) / \log \lambda $, we have $ \left\| A^n(\omega) \right\| > 2$, so using \eqref{e.sing.mnfld.pert} gives us
\begin{align*}
d \left( \Lambda_n^s(\omega), \Lambda_{n+1}^s(\omega) \right)
& \leq
\frac{\pi}{2} \sqrt{\frac{B^2 \|A^{n+1}(\omega) \|^{-2} - \|A^n(\omega)\|^{-2}}{\|A^n(\omega)\|^2 - \|A^n(\omega)\|^{-2}}}\\
& \leq
\pi B \|A^n(\omega) \|^{-1} \| A^{n+1}(\omega) \|^{-1} \\
& \leq
\pi B^2 \|A^n(\omega) \|^{-2}.
\end{align*}
We have used the easy inequality $\left(x^2 - x^{-2} \right)^{-1} \leq 4 x^{-2}$ for $x \geq 2$ in the second line.  Thus, we have shown
\begin{equation} \label{e.sing.mnflds.cauchy}
d \left( \Lambda_n^s(\omega), \Lambda_{n+1}^s(\omega) \right)
\lesssim
\lambda^{-2n},
\end{equation}
which implies that $(\Lambda_n^s(\omega))_{n=1}^{\infty} $ is a Cauchy sequence in $\C\PP^1$ for each $\omega$, so $ \Lambda^s(\omega) = \lim_{n \to \infty} \Lambda_n^s(\omega) $ exists for each $\omega$.  Of course, \eqref{e.sing.mnflds.cauchy} implies
\begin{align*}
d \left(\Lambda_n^s(\omega)), \Lambda^s(\omega) \right)
& \leq
\sum_{m=n}^{\infty} d \left(\Lambda_m^s(\omega),\Lambda_{m+1}^s(\omega) \right) \\
& \lesssim
\sum_{m = n}^{\infty} \lambda^{-2m}  \\
& \lesssim
\lambda^{-2n},
\end{align*}
so the sequence $\Lambda_n^s(\omega)$ converges uniformly on $\Omega$ -- in particular, $\Lambda^s(\cdot)$ is continuous, since \eqref{e.sing.mnfld.pert} implies that the map $ \omega \mapsto \Lambda_n^s(\omega)$ is continuous for each $n > \log\left(C^{-1}\right)/\log\lambda $.

Next, we prove invariance.  Notice that
$$
\|A^n(T\omega) A(\omega) v_{n+1}(\omega) \|
=
\| A^{n+1}(\omega) v_{n+1}(\omega) \|
=
\| A^{n+1}(\omega) \|^{-1}
\lesssim
\lambda^{-n},
$$
so another application of Lemma \ref{l.sing.mnflds} yields
$$
d \left( A(\omega) \cdot \Lambda_{n+1}^s(\omega), \Lambda_n^s(T\omega) \right)
\lesssim
\lambda^{-2n}.
$$
Passing to the limit, we obtain
$$
A(\omega) \cdot \Lambda^s(\omega) = \Lambda^s(T \omega),
$$
as desired.  Using similar arguments on the left half-line, we may construct
$$
\Lambda^u(\omega) = \lim_{n \to -\infty} \Lambda^s_n(\omega),
$$
which satisfies $A(\omega) \cdot \Lambda^u(\omega) = \Lambda^u(T\omega)$ and
\begin{equation} \label{e.unstable.mnfld.convergence}
d\left(\Lambda^s_{-n}(\omega), \Lambda^u(\omega) \right)
\lesssim
\lambda^{-2n}.
\end{equation}
Throughout the remainder of the proof, $ v_s(\omega) $ and $v_u(\omega)$ will denote unit vectors in $\Lambda^s(\omega)$ and $\Lambda^u(\omega)$, respectively.

The next step is to prove $ \Lambda^s(\omega) \neq \Lambda^u(\omega) $ for every $\omega$.  We will accomplish this by proving that there exist $n, k \in \Z$ such that $A^n(T^k \omega)$ shrinks vectors in $\Lambda^s(T^k \omega)$ and expands vectors in $\Lambda^u(T^k \omega)$.  To do this, we need to find a growth rate for cocycle iterates which is sufficiently close to optimal along a given orbit.  To that end, fix $\omega_0 \in \Omega $, and consider the set $M = M(\omega_0)$ defined by
$$
M
=
 \left\{
r > 1 : \text{for all but finitely many } n \in \Z, \,  \|A^n(T^k\omega_0)\| \geq C r^{|n|} \text{ for all } k \in \Z
\right\},
$$
and put $R = R(\omega_0) = (\sup M)^{3/4}$. Note that the $C$ in the definition of $M$ is the $C$ from \eqref{e.exp.growth} so that the set on the right hand side is nonempty.  Since $R < \sup M$, there exists $C'>0$ such that
$$
\| A^n(T^k \omega_0) \|
\geq
C ' R^{|n|}
$$
for all $n,k \in \Z$ (in particular, $C'$ depends on $\omega_0$, but \emph{not} on $k$). More explicitly, since $R \in M$, there is a finite set $F \subseteq \Z$ such that
$$
\| A^n(T^k \omega_0) \| \geq CR^{|n|}
$$
for all $n \in \Z \setminus F$ and all $k \in \Z$. One can then take
$$
C'
=
\min
\left(C , R^{-N}
\right),
$$
where $N = \sup_{n \in F} |n|$. However, because $R^{3/2} > \sup M$, we have
$$
\left\| A^{n_j} \left( T^{k_j}\omega_0 \right) \right\|
\leq
C R^{3 |n_j| / 2}
$$
for sequences $ n_j,k_j \in \Z $ such that $ |n_j| \to \infty $. Without loss of generality, assume $n_j \to \infty $.  By running the argument that we used to construct $\Lambda^s$, we see that
$$
\phi_j
:=
d \left( \Lambda^s \left( T^{k_j}\omega_0 \right),\Lambda_{n_j}^s \left( T^{k_j}\omega_0 \right) \right)
\lesssim
R^{-2n_j},
$$
where we have used that $C'$ is $k$-independent to obtain a $j$-independent implicit constant on the right hand side.  As a consequence,
\begin{align*}
\left\| A^{n_j} \left( T^{k_j}\omega_0 \right) v_s \left( T^{k_j} \omega_0 \right) \right\|^2
& =
\left\| A^{n_j} \left( T^{k_j} \omega_0 \right) \right\|^{2} \sin^2(\phi_j) + \left\| A^{n_j} \left( T^{k_j}\omega_0 \right) \right\|^{-2}\cos^2(\phi_j) \\
& \lesssim
R^{3n_j} (R^{-2n_j})^2 + R^{-2 n_j},
\end{align*}
where we have used the orthogonality statement from Lemma \ref{l.sing.mnflds} to obtain the first line. In particular, this bound implies that $\| A^{n_j}(T^{k_j} \omega_0) v_s(T^{k_j}\omega_0) \| < 1$ for sufficiently large $j$.  Similarly, one can show that
\begin{equation} \label{eq:rev:decay}
\left\| A^{-n_j} \left( T^{k_j+n_j}\omega_0 \right) v_u \left( T^{k_j+n_j}\omega_0 \right) \right\| < 1
\end{equation}
for large $j$.  More precisely, one can use \eqref{eq:cocyc:rev} to see that
$$
\left\| A^{-n_j}\left(T^{k_j+n_j}\omega_0 \right) \right\|
=
\left\| A^{n_j}(T^{k_j}\omega_0)^{-1} \right\|
=
\left\| A^{n_j}(T^{k_j}\omega_0) \right\|,
 $$
where we have used $|\det|=1$ in the second equality.  With this observation, one can run an argument almost identical to the one above to prove \eqref{eq:rev:decay} for sufficiently large $j$.  It follows that
$$
\left\| A^{n_j} (T^{k_j}\omega_0) v_u \left( T^{k_j}\omega_0 \right) \right\| > 1
$$
for such $j$.  Thus, $ \Lambda^s(T^{k_j}\omega_0) \neq \Lambda^u(T^{k_j}\omega_0) $ for some large $j$.  By invariance, $ \Lambda^s(\omega_0) \neq \Lambda^u(\omega_0) $. Consequently, by compactness of $\Omega$ and continuity of $\Lambda^s$ and $\Lambda^u$, there exists $ \gamma > 0 $ so that $ d \left( \Lambda^s(\omega), \Lambda^u(\omega) \right) \geq \gamma $ for all $\omega \in \Omega $.

To see that we have the desired exponential decay estimates, put
$$
\theta_n
=
\theta_n(\omega)
=
d \left(\Lambda^s \left(T^n\omega \right), \Lambda_{-n}^s(T^n \omega) \right).
$$
Now, recall that $ d(\Lambda_{-n}^s(T^n \omega), \Lambda^u(T^n\omega)) \lesssim \lambda^{-2n} $ by \eqref{e.unstable.mnfld.convergence}, so
$$
\theta_n(\omega) \geq \gamma - C_1\lambda^{-2n}
$$
for some constant $C_1 > 0$, all $n \in \Z_+ $, and every $\omega \in \Omega$. Thus,
\begin{align*}
\| A^{-n}(T^n\omega) v_s(T^n \omega) \|^2
& =
\| A^{-n}(T^n\omega) \|^2 \sin^2(\theta_n) + \| A^{-n}(T^n\omega) \|^{-2} \cos^2(\theta_n) \\
& \gtrsim
\lambda^{2n} \sin^2(\gamma/2) \\
& \gtrsim
\lambda^{2n}.
\end{align*}
Using invariance and applying the cocycle identity \eqref{eq:cocyc:rev} to the previous bound, we obtain
$$
\| A^{n}(\omega) v_s(\omega) \| \lesssim  \lambda^{-n},
$$
as desired.  One proves the decay estimates for $v_u(\omega)$ in a similar fashion.

(b) $\implies$ (c) Suppose $(T,A)$ admits an invariant exponential splitting $\Lambda^s,\Lambda^u$ with associated constants $C>0,\lambda>1$.  Given $ \omega \in \Omega $ and $v \in \S^3$, one can write $ v = c_s v_s(\omega) + c_u v_u(\omega) $.  If $c_s \neq 0 $ (respectively $c_u \neq 0$), then $\|A^n(\omega)v\|$ grows exponentially quickly on the left half-line (respectively the right half-line).  In either case, it is clear that $v$ is not a Sacker-Sell solution.

(c) $\implies $ (a)  Consider the open sets
$$
O_{n,\delta}
=
\left\{ (\omega,v) \in \Omega \times \S^3 : \| A^n(\omega) v \| > 1+\delta \right\}.
$$
By (c), the family $ \{O_{n,\delta} : n \in \Z, \delta>0\} $ constitutes an open cover of $\Omega \times \S^3$, so, by using compactness to pass to a finite subcover, we see that we can choose $N \in \Z_+$ and $\varepsilon > 0$ such that for any $(\omega,v) \in \Omega \times \S^3$, there exists $n$ with $|n| \leq N$ and
\begin{equation} \label{e.matnorm.bounded.below}
\| A^n(\omega) v \| > 1 + \varepsilon.
\end{equation}
To prove exponential bounds of the form \eqref{e.exp.growth}, it suffices to prove the following two claims.

\vspace*{4pt}\noindent \textbf{Claim 1.}  \emph{For every $ \omega \in \Omega $ and every $ v \in \S^3 $, there exists a sequence $ n_k \in \Z $ with the following properties:}
\begin{enumerate}
\item \emph{$n_k$ is strictly \emph{monotone}.}
\item \emph{$ n_0 = 0 $.}
\item \emph{$ |n_{k} - n_{k-1}| \leq N $ for all $ k \in \Z_+ $.}
\item \emph{$ \| A^{n_k}(\omega) v \| > (1 + \varepsilon)^k $ for every $ k \in \Z_+ $.}
\end{enumerate}

\vspace*{4pt}\noindent \textbf{Claim 2.}  \emph{For every $ \omega \in \Omega $, there exists $ v_+ = v_+(\omega) \in \S^3 $ and a strictly \emph{increasing} sequence $n_k \in \Z $ which satisfies items (2)--(4) of Claim 1.}

\vspace*{4pt}
First, let us see how Claim 2 implies the desired result.  Given $ \omega \in \Omega $, let $ v_+ \in \S^3 $ be as in Claim 2.  Obviously, items (3) and (4) imply that
\[
\|A^{n_k}(\omega) \|
>
(1+\varepsilon)^k
\geq
(1 + \varepsilon)^{n_k/N}
\]
Recall $B = \max\|A(\omega)\| $.  To obtain exponential lower bounds for every iterate on the right half-line, one simply interpolates and uses uniform boundedness of $|n_k - n_{k-1}|$.  More precisely, fix $k \in \Z_+$ and $ 1 \leq r \leq N $.  We have
\begin{align*}
\| A^{n_k+r}(\omega) \|
& \geq
\|A^{n_k}(\omega) \| B^{-r} \\
& \geq
(1+\varepsilon)^k B^{-r} \\
& \geq
\frac{(1+\varepsilon)^{k+r/N}}{B^N(1+\varepsilon)} \\
& \geq
\frac{(1+\varepsilon)^{\frac{1}{N}(n_k+r)}}{B^N(1+\varepsilon)}.
\end{align*}
In particular, with  $ \lambda = (1+\varepsilon)^{1/N} $ and $ C = B^{-N} (1+\varepsilon)^{-1} $, we have an exponential lower bound of the form \eqref{e.exp.growth} for all $ \omega \in \Omega $ and $ n \geq 0 $.  To extend this to negative $n$, simply apply the cocycle identity \eqref{eq:cocyc:rev} to deduce
\[
\left\| A^{-n}(\omega) \right\|
=
\left\| A^n(T^{-n} \omega)^{-1} \right\|
=
\left\| A^n(T^{-n}\omega) \right\|
\geq
C\lambda^n
\]
for every $ n \in \Z_+ $ (note that the second equality uses $|\det|=1$).  Thus, we need only prove the claims.

\begin{proof}[Proof of Claim 1]

Let $\omega\in\Omega, v \in \S^3$ be given, and apply \eqref{e.matnorm.bounded.below} inductively  to choose integers $n_1,n_2,\ldots $ so that $ |n_k - n_{k-1}| \leq N $ and
$$
 \| A^{n_k}(\omega) v \| \geq (1+\varepsilon)^k
$$
for all $k$.  Since  $(1+\varepsilon)^k \to \infty$ as $k \to \infty$, we must have either $ n_k \to + \infty $ or $n_k \to - \infty$ as $k \to \infty$.  Without loss of generality, assume that $ n_k \to +\infty $ (the argument in the case $n_k \to -\infty $ is nearly identical).  If $0 < k_1 < k_2 < \cdots$ are positive integers, we make the trivial observation that
\[
\| A^{n_{k_j}}(\omega) v \|
\geq
(1+\varepsilon)^{k_j}
\geq
(1+\varepsilon)^j.
\]
As a consequence, we may pass to a subsequence to produce $n_k$ which are strictly monotone and still satisfy items (2)--(4) of the claim.
\end{proof}

\begin{proof}[Proof of Claim 2]
Fix $ \omega \in \Omega $  and suppose for the sake of establishing a contradiction that there does \emph{not} exist a $v_+ \in \S^3 $ and an increasing sequence $ n_k \in \Z_+ $ which satisfies items (2)--(4). Then, for every $v \in \S^3$, Claim 1 implies that $v$ enjoys a strictly \emph{decreasing} sequence $ n_k \in \Z $ with $n_0 = 0$, $ |n_k - n_{k-1}| \leq N $, and $ \| A^{n_k}(\omega)v\| > (1+ \varepsilon)^k $ for all $k$.  Interpolating as before, we obtain estimates of the form
\begin{equation} \label{eq:uh:char:impossible:lb}
\| A^{-n}(\omega) v \|
>
\tilde C (1 + \varepsilon)^{n/N}
\end{equation}
for some $ \tilde C>0$ which is \emph{uniform} in $v \in \S^3 $ and $n \in \Z_+$.  However, this is nonsense: for sufficiently large $n$, \eqref{eq:uh:char:impossible:lb} implies that $ \|A^{-n}(\omega) v\| >1 $ for every $ v \in \S^3 $, which contradicts unimodularity of $\det\left(A^{-n}(\omega)\right)$.  This proves the claim.
\end{proof}

This completes the proof of Theorem \ref{t.uh.splitting}.
\end{proof}

Item (c) in the characterization above is pleasant for at least two reasons.  First, it allows us to relate (the absence of) hyperbolicity of cocycles to existence of bounded orbits of the cocycle action, which provides a link between hyperbolicity and the theory of generalized eigenfunctions.  Second, it makes robustness of uniform hyperbolicity under perturbations very easy to see.

\begin{corollary} \label{c:uh:open}
Uniform hyperbolicity is a robust property in the sense that $ \UH $ defines an open subset of $ \mathrm{Homeo}(\Omega) \times C(\Omega, G) $, where both factors are endowed with the appropriate uniform topology.
\end{corollary}

\begin{proof}
Given a uniformly hyperbolic cocycle $ (T,A) $, use the compactness argument from the proof of (c) $\implies$ (a) to produce $ \varepsilon > 0$, $N \in \Z_+ $ so that for any $ \omega \in \Omega ,v \in \X^3 $, there is $|n| \leq N $ so that $ \|A^n(\omega)v\| > 1+\varepsilon $.  It is then easy to see that if $T'$ and $ A' $ are sufficiently small uniform perturbations of $T$ and $A$, then $ \left( T', A' \right) $ will also enjoy this property.
\end{proof}

\section{Generalized eigenfunctions of CMV matrices} \label{s:genef}

Suppose $(\alpha_n)_{n \in \Z} \in \D^{\Z}$ and the associated extended CMV matrix $\mathcal{E}$ are given. For simplicity of notation, we introduce
\begin{align*}
a_n & = -\overline{\alpha_n} \alpha_{n-1}\\
b_n & = \overline{\alpha_n} \rho_{n-1}\\
c_n & = -\rho_n \alpha_{n-1}\\
d_n & = \rho_n \rho_{n-1}
\end{align*}
for each $n \in \Z$. In terms of these parameters, the matrix representation of $\E$ becomes
$$
\E
=
\begin{pmatrix}
\ddots & \ddots & \ddots &&&&&  \\
 & a_0 & b_1 & d_1 &&& & \\
& c_0 & a_1 & c_1 &&& & \\
&  & b_2 & a_2 & b_3 & d_3 & & \\
& & d_2 & c_2 & a_3 & c_3 &  &  \\
& &&& b_4 & a_4 & b_5 & \\
& &&& d_4 & c_4 & a_5 &   \\
& &&&& \ddots & \ddots &  \ddots
\end{pmatrix}.
$$

\begin{definition} \label{d:genef} \normalfont
 We say that $\phi$ is a \emph{generalized eigenvector} of $\E$ with corresponding \emph{generalized eigenvalue} $z$ if $\phi:\Z \to \C$ is a nonzero sequence which satisfies
$
\E  \phi = z \phi
$
and is polynomially bounded.  That is, there exist $R,S > 0$ with
$$
|\phi_n|
\leq
R \big( 1 + |n| \big)^S
$$
for all $n \in \Z$.  Notice that $\phi$ is not necessarily an element of $\ell^2(\Z)$.
\end{definition}

The following lemma is standard and easily proved by an inductive argument.  This is done in \cite[Lemma 3]{MO}, for example.

\begin{lemma} \label{l:cmv.spectral.basis}
The set $\{\delta_0,\delta_1\}$ is a spectral basis for $\E$.  That is to say, the set of finite linear combinations of vectors of the form $\E^k \delta_0$ and $\E^k \delta_1$ with $k \geq 0$ is dense in $\ell^2(\Z)$.
\end{lemma}

\begin{definition} \normalfont
Let $\E$ be an extended CMV matrix.  For $\psi \in \ell^{2}(\Z)$, define the spectral measure $\mu_\psi^{\E}$ by
\[
\langle \psi, g(\E) \psi \rangle
=
\int_{\partial \D} \! g(z) \, d\mu_\psi^{\E}(z)
\]
for all continuous functions $g:\partial \D \to \C$.
  By Lemma \ref{l:cmv.spectral.basis}, $\mu := \mu_{\delta_0}^{\E} + \mu_{\delta_1}^{\E}$ is a universal spectral measure for $\E$ in the sense that all other spectral measures of $\E$ are absolutely continuous with respect to $\mu$.  We call $\mu$ the \emph{canonical spectral measure} of $\E$.
\end{definition}

\begin{theorem}[Schnol's Theorem for Extended CMV Matrices] \label{t:schnol:cmv}
Let $\E$ be an extended CMV matrix with Verblunsky coefficients $(\alpha_n)_{n \in \Z}$, and let $\mathcal G$ denote the set of generalized eigenvalues of $\E$. Then,
\begin{itemize}
\item[\rm (a)] $\mathcal G \subseteq \sigma(\E)$,
\item[\rm (b)] $\mu(\sigma(\E) \setminus \mathcal G) = 0$,
\item[\rm (c)] $\overline{\mathcal G} = \sigma(\E)$.
\end{itemize}
\end{theorem}

\begin{proof}
(a) Suppose $z \in \mathcal G$ and let $\phi$ be a corresponding generalized eigenfunction.  As in the Schr\"odinger case, the main idea is that normalized cutoffs of $\phi$ will produce a Weyl sequence because the ``integral'' dominates the boundary terms for polynomially bounded sequence.  More precisely, for each $N \in \Z_+$, define $\phi^{(N)}, \psi^{(N)} \in \ell^2(\Z)$ by
$$
\phi^{(N)}_k
=
\begin{cases}
\phi(k) & -2N + 1 \leq k \leq 2N \\
0 & \text{otherwise}
\end{cases}
\quad
\text{and}
\quad
\psi^{(N)}
=
\frac{\phi^{(N)}}{ \left\| \phi^{(N)} \right\|}.
$$
Since $\phi$ and the zero sequence are both in the kernel of $(\E - z)$, we have\break $ \left( (\E - z)\phi^{(N)} \right)_k = 0 $ unless $k = -2N,-2N+1,2N,2N+1$. Calculations at those points yield
\begin{align}
\label{eq:genef.cutoff0}
\left[ (\E - z) \phi^{(N)} \right]_{-2N}
& =
b_{-2N+1} \phi_{-2N+1} + d_{-2N+1} \phi_{-2N+2}
\\
\label{eq:genef.cutoff1}
\left[ (\E - z) \phi^{(N)} \right]_{-2N+1}
& =
-d_{-2N} \phi_{-2N-1} - c_{-2N} \phi_{-2N}
\\
\label{eq:genef.cutoff2}
 \left[ (\E - z) \phi^{(N)} \right]_{2N}
& =
-b_{2N+1} \phi_{2N+1} - d_{2N+1} \phi_{2N+2}
\\
\label{eq:genef.cutoff3}
 \left[ (\E - z) \phi^{(N)} \right]_{2N+1}
& =
d_{2N}\phi_{2N-1} + c_{2N} \phi_{2N}
\end{align}
Note that \eqref{eq:genef.cutoff1} and \eqref{eq:genef.cutoff2} use the fact that $\phi$ satisfies $\E \phi = z \phi$. In particular, \eqref{eq:genef.cutoff0}--\eqref{eq:genef.cutoff3} yield
$$
\frac{1}{2} \left\| (\E - z) \phi^{(N)} \right\|^2
\leq
\sum_{j=-1}^2  | \phi_{-2N+j}|^2 +| \phi_{2N+j}|^2
=
 \left\| \phi^{(N+1)} \right\|^2 - \left\| \phi^{(N-1)} \right\|^2.
$$
Now, we claim that
\begin{equation} \label{eq:genef:weyl}
\liminf_{N \to \infty}
\left\| (\E - z) \psi^{(N)} \right\|
=
0.
\end{equation}
If \eqref{eq:genef:weyl} fails, then there exist $N_0 \in \Z_+$, $\delta > 0$ so that
$$
\left\| (\E - z) \psi^{(N)} \right\|^2
\geq
\delta
$$
for $N \geq N_0$.  Thus, for $N \geq N_0$,
$$
\left\| \phi^{(N+1)} \right\|^2
\geq
\left( 1 + \frac{\delta}{2} \right) \left\| \phi^{(N-1)} \right\|^2,
$$
which leads to an exponential lower bound on the growth of $ \left\| \phi^{(N)} \right\|^2 $, contradicting polynomial boundedness of the same. Thus, \eqref{eq:genef:weyl} holds, and so $z \in \sigma(\E)$ follows readily from the Weyl criterion.

(b) For each $n,m \in \Z$, define a (complex) Borel measure $\mu_{n,m}$ on $\partial \D$ via
$$
\mu_{n,m}(B)
=
\langle \delta_n, \chi_B(\E) \delta_m \rangle
$$
Define
$$
c =
\left( \sum_{n \in \Z} ( 1 + |n| )^{-2} \right)^{-1}
,
\quad
\lambda_n
=
c(1+|n|)^{-2},
\quad
\rho
=
\sum_{n \in \Z} \lambda_n \mu_{n,n}.
$$
Evidently, $\rho$ is a probability measure and and $\mu \ll \rho \ll \mu$.  For each $n,m \in \Z$, let $ \Phi_{n,m}$ denote the Radon-Nikodym derivative of $\mu_{n,m}$ with respect to $\rho$.  That is to say
$$
\int_{\partial \D} \! g(z) \, d\mu_{n,m}(z)
=
\int_{\partial \D} \! g(z) \Phi_{n,m}(z) \, d\rho(z)
$$
for all continuous functions $g$ on $\partial \D$.
For each $z \in \partial \D$, define a sequence $\varphi^z$ via $ \varphi^z_n = \Phi_{n,0}(z)$.  Notice first that
\begin{align*}
\rho(B)
& =
\sum_{n \in \Z} \lambda_n \mu_{n,n}(B) \\
& =
\int_{\partial \D} \! \chi_B(z) \left( \sum_{n \in \Z} \lambda_n \Phi_{n,n}(z) \right) \, d\rho(z)
\end{align*}
for each Borel set $B \subseteq \partial \D$.  Thus, $\Phi_{n,n}(z) \leq \lambda_n^{-1}$ for every $n \in \Z$ and $\rho$ almost every $z$.  Consequently, by Cauchy-Schwarz,
\begin{align*}
\left| \int \! \chi_B(z) \Phi_{n,m}(z) \, d\rho(z) \right|
& =
\left| \mu_{n,m}(B) \right| \\
& \leq
\mu_{n,n}(B)^{1/2} \mu_{m,m}(B)^{1/2} \\
& \leq
\lambda_n^{-1/2} \lambda_m^{-1/2} \rho(B).
\end{align*}
We then must have
$$
|\Phi_{n,m}(z)|
\leq
\lambda_n^{-1/2} \lambda_m^{-1/2}
$$
for all $n,m \in \Z$ and $\rho$ almost every $z \in \partial \D$.  In particular, $ |\varphi_n^z| \lesssim (1+|n|) $ for $\rho$ almost every $z$.  For even $n$ and any continuous $g$, we have
\begin{align*}
\int \! g(z) (z \varphi_n^z) \, d\rho(z)
& =
\int \! zg(z) \, d\mu_{n,0}(z) \\
& =
\langle \delta_n, \E g(\E) \delta_0 \rangle \\
& =
\langle \E^* \delta_n, g(\E) \delta_0 \rangle \\
& =
\langle b_n^* \delta_{n-1} + a_n^* \delta_n + b_{n+1}^* \delta_{n+1} + d_{n+1}\delta_{n+2} , g(\E) \delta_0 \rangle \\
& =
\int \! \big(  b_n \varphi_{n-1}^z + a_n \varphi_n^z + b_{n+1} \varphi_{n+1}^z + d_{n+1}\varphi_{n+2}^z \big) \, g(z) \, d\rho(z) \\
& =
\int \! (\E \varphi^z)_n \, g(z) \, d\rho(z).
\end{align*}
A similar calculation works for odd $n$.  Thus, $ \varphi^z $ is a generalized eigenfunction of $\E$ for $\rho$ almost every (hence $\mu$ almost every) $z \in \partial \D$.

(c) By (a) and (b), $\overline{\mathcal G}$ is a closed subset of $\sigma(\E)$ which supports $\mu$. Since every spectral measure is absolutely continuous with respect to $\mu$, $\mathcal G$ supports every spectral measure.  Since the spectrum of $\E$ is the smallest closed set which supports every spectral measure of $\E$, the conclusion in (c) follows.
\end{proof}

With these pieces in place, we can combine Theorems \ref{t.uh.splitting} and \ref{t:schnol:cmv} to describe the spectra of dynamically defined extended CMV matrices in terms of the region of energies for which the associated cocycles are uniformly hyperbolic.  More precisely, if $z$ is such that one (and hence both) of the associated cocycles is uniformly hyperbolic, then $z$ is not in the spectrum of any operator in the family.  If the family fibers over minimal dynamics, then the converse holds as well.

Since this theorem is most transparently proved by way of generalized eigenfunctions, we need a matrix cocycle associated to an extended CMV matrix which encodes the behavior of solutions to the difference equation $\E u = zu$.  To that end, consider the  matrices
\[
P(\alpha,z)
=
\frac{1}{\rho}
\begin{pmatrix}
-\alpha & z^{-1} \\
z & - \overline\alpha
\end{pmatrix}
,
\quad
Q(\alpha,z)
=
\frac{1}{\rho}
\begin{pmatrix}
-\overline\alpha & 1 \\
1 & - \alpha
\end{pmatrix}.
\]
These matrices were introduced in \cite{GZ06} by Gesztesy and Zinchenko, although we use the normalization in \cite{O14} -- that is, we replace $\alpha$ by $-\overline{\alpha}$ in Gesztesy-Zinchenko's definition of a CMV matrix, which is then reflected here.

Let us describe how the Gesztesy-Zinchenko matrices capture the recursion described by the difference equation $ \E u = zu $. First, recall that an extended CMV matrix $\E$ enjoys a factorization of the form $ \E = \mathcal{L} \mathcal{M}$, where $\mathcal L$ and $\mathcal M$ are direct sums of $2 \times 2$ matrices of the form
$$
\Theta(\alpha)
=
\begin{pmatrix}
\overline{\alpha} & \rho \\
\rho & - \alpha
\end{pmatrix}.
$$
More precisely, one may take
\[
\mathcal L
=
\bigoplus_{n \in \Z} \Theta(\alpha_{2n}),
\quad
\mathcal M
=
\bigoplus_{n \in \Z} \Theta(\alpha_{2n+1}),
\]
where $\Theta(\alpha_j)$ is understood to act on coordinates $j$ and $j+1$ in both direct sums; compare \cite[Proposition~4.2.4]{S05}. Here, we follow the conventions of \cite{S05} (\cite{GZ06} has $\Theta(\alpha_j)$ act on coordinate $j-1$ and $j$). Now, if $u$ is a complex sequence such that $ \E u = zu $, then, with $ v = \mathcal{M}u $, it is easy to check that $ \E^{\top} \! v = zv $. Moreover, from the explicit form of $\mathcal{L}$ and $\mathcal{M} $ it is apparent that $u$ is bounded (resp., polynomially bounded) if and only if $v$ is bounded (resp., polynomially bounded). Gesztesy and Zinchenko's arguments show that
$$
\begin{pmatrix}
u_{n+1} \\ v_{n+1}
\end{pmatrix}
=
Y(n,z)
\begin{pmatrix}
u_n \\ v_n
\end{pmatrix}
$$
for all $n$, where $Y(n,z) = P(\alpha_n,z)$ when $n$ is even and $Y(n,z) = Q(\alpha_n,z)$ when $n$ is odd.

It is helpful to note that the Gesztesy-Zinchenko and Szeg\H{o} matrices are closely related. Indeed, we have
$$
Q(\alpha,z) P(\beta,z)
=
\frac{1}{\rho_\alpha \rho_\beta}
\begin{pmatrix}
z+\overline{\alpha}\beta &  - \overline{\beta} -\overline{\alpha} z^{-1} \\
- \alpha z  - \beta  & \alpha \overline{\beta} + z^{-1}
\end{pmatrix},
$$
while
$$
S(\alpha,z) S(\beta,z)
=
\frac{1}{\rho_\alpha \rho_\beta}
\begin{pmatrix}
z^2 + \overline{\alpha} \beta z & -\overline{\beta}z - \overline{\alpha} \\
-\alpha z^2 - \beta z & \alpha \overline{\beta} z+ 1
\end{pmatrix}.
$$
In particular, we notice that
\begin{equation}\label{eq:szego:gz:rel}
S(\alpha,z) S(\beta,z)
=
z
Q(\alpha,z) P(\beta,z),
\end{equation}
for all $\alpha, \beta \in \D$ and all $z \in \C \setminus \{0\}$.

As in the introduction, let $ \Omega $ be compact, $T:\Omega \to \Omega$ a homeomorphism, and $f:\Omega \to \D$ continuous.  We obtain a family of extended CMV matrices $ (\E_\omega)_{\omega\in \Omega} $ where $\E_\omega$ has $ \alpha(n) = \alpha_\omega(n) = f(T^n\omega) $. There are then two families of cocycles which one can naturally associate to this family of operators, namely, the Szeg\H{o} cocycle, which was described in the introduction, and the Gesztesy-Zinchenko cocycle, which is induced by the Gesztesy-Zinchenko transfer matrices.  There is a minor annoyance, in that the Gesztesy-Zinchenko matrices alternate, but this is easily resolved by passing to blocks of length two.  More precisely, for $z \in \partial \D$, we define the Gesztesy-Zinchenko cocycle by
\[
G_z(\omega)
=
Q(f(T\omega),z) P(f(\omega),z)
=
Q( \alpha_\omega(1),z) P(\alpha_\omega(0),z).
\]

Given \eqref{eq:szego:gz:rel}, the following theorem is immediate.

\begin{theorem} \label{t:uh:equiv}
Let $\Omega$, $T$, and $f$ be as above and let $z \in \partial \D$ be given.  The following are equivalent.
\begin{enumerate}
\item $(T,A_z)$ is uniformly hyperbolic.
\item $(T^2,A_z^2)$ is uniformly hyperbolic.
\item $(T^2,G_z)$ is uniformly hyperbolic.
\end{enumerate}
\end{theorem}

\begin{proposition} \label{p:uh:spec:1}
Denote by $U$ the set of $z \in \partial \D $ so that the cocycles above are uniformly hyperbolic. Then $ \sigma(\E_\omega) \subseteq \partial \D \setminus U $ for every $\omega \in \Omega$.
\end{proposition}

\begin{proof}
Suppose $ (T^2,G_z) $ is uniformly hyperbolic. Fix $\omega \in \Omega$, suppose $u$ is a nontrivial solution to $ \E_\omega u = z u $, and set $v = \mathcal{M}_\omega u$. By uniform hyperbolicity, the norm of the vector $(u_n,v_n)$ grows exponentially fast on at least one half-line.  Since $u$ is polynomially bounded if and only if $v$ is polynomially bounded, it follows that $u$ is not a generalized eigenvector of $\E_\omega$. Consequently, $z \notin \mathcal{G}_\omega $, the set of generalized eigenvalues of $\E_\omega$.  Since $U$ is open by Corollary \ref{c:uh:open}, we have
$$
\sigma(\E_\omega)
=
\overline{\mathcal{G}_\omega}
\subseteq
\partial \D \setminus U,
$$
by Theorem \ref{t:schnol:cmv}(c).
\end{proof}

As a result of these considerations, we recover (by very elementary means) the Geronimo-Johnson theorem.

\begin{proof}[Proof of Theorem~\ref{t:minspec:johnson}]
The first claim is well-known and not hard to prove by a strong operator approximation argument using \cite[Theorem VIII.24]{reedsimon1}, for example.  By Proposition~\ref{p:uh:spec:1}, it suffices to prove $\partial \D \setminus U \subseteq \Sigma$ to complete the proof.  To that end, suppose $(T^2,G_z)$ is not uniformly hyperbolic. By Theorem \ref{t.uh.splitting}(c), there exist $\omega \in \Omega$ and $\textbf{u} \in \S^3 $ so that
$
\left\| G_z^n(\omega)
\textbf{u}
\right\|
\leq
1
$
for all $n \in \Z$.  Define sequences $u$ and $v$ by
$$
\begin{pmatrix}
u_{2n} \\ v_{2n}
\end{pmatrix}
:=
G_z^n(\omega) \textbf{u},
\quad
\begin{pmatrix}
u_{2n+1} \\ v_{2n+1}
\end{pmatrix}
:=
P(\alpha_{\omega}(2n),z) G_z^n(\omega) \textbf{u},
\quad
n \geq 0,
$$
and similar formulae on the left half line.  It is easy to see that $(u_n)_{n \in \Z}$ is a generalized eigenfunction of $\E_\omega$, which implies $z \in \sigma(\E_\omega) = \Sigma$.
\end{proof}

\vspace*{4pt}
\noindent \textbf{Remark.}  Theorem~\ref{t.uh.splitting} requires no assumptions on $T$. In particular, Theorem~\ref{t.uh.splitting} holds even when $T$ is not minimal. This comes into play indirectly in Theorem~\ref{t:uh:equiv}, for, if minimality were required, then we would need to explicitly assume minimality of both $T$ and $T^2$. In fact, the only place where we have used minimality is to get $\omega$-invariance of the spectrum. In general, if $T$ is not minimal, one can follow the arguments above to see that $\sigma(\E_\omega) = \partial \D \setminus U$ for any $\omega$ whose $T$-orbit is dense in $\Omega$.
\section*{Acknowledgments} J.\ F.\ wishes to thank Chris Marx and Ian Morris for helpful discussions and suggestions.


\begin{thebibliography}{99}

\bibitem{B} 
 Ju.\ M.\ Berezanskii,
 \emph{Expansions in Eigenfuncions of Selfadjoint Operators},
 Amer.\ Math.\ Soc., Providence, 1968.

\bibitem{BG} 
 J.\ Bochi and N.\ Gourmelon,
 {Some characterizations of domination},
 \emph{Math.\ Z.,}\ \textbf{263} (2009), 221--231.

\bibitem{DFLY} 
 D.\ Damanik, J.\ Fillman, M.\ Lukic and W.\ Yessen,
 {Uniform hyperbolicity for Szeg\H{o} cocycles and applications to random CMV matrices and the Ising model},
 \emph{Int.\ Math.\ Res.\ Not.,}\ \textbf{2015} (2015), 7110--7129.

\bibitem{DFO} 
 D.\ Damanik, J.\ Fillman and D.\ C.\ Ong,
 {Spreading estimates for quantum walks on the integer lattice via power-law bounds on transfer matrices},
 \emph{J.\ Math.\ Pures Appl.,}\ \textbf{105} (2016), 293--341.

\bibitem{GJ96} 
 J.\ Geronimo and R.\ Johnson,
 {Rotation number associated with difference equations satisfied by polynomials orthogonal on the unit circle},
 \emph{J.\ Differential Equations,} \textbf{132} (1996), 140--178.

\bibitem{GZ06} 
 F.\ Gesztesy and M.\ Zinchenko,
 {Weyl-Titchmarsh theory for CMV operators associated with orthogonal polynomials on the unit circle},
 \emph{J.\ Approx.\ Theory,} \textbf{139} (2006), 172--213.

\bibitem{J} 
 R.\ Johnson,
 {Exponential dichotomy, rotation number, and linear differential operators with bounded coefficients},
 \emph{J.\ Diff.\ Eq.,}\ \textbf{61} (1986), 54--78.

\bibitem{LS99}
 Y.\ Last and B.\ Simon,
 {Eigenfunctions, transfer matrices, and absolutely continuous spectrum of one-dimensional Schr\"odinger operators},
 \emph{Invent.\ Math.,}\ \textbf{135} (1999), 329--367.

\bibitem{LO}
 M.\ Lukic and D.\ Ong,
 Generalized Pr\"ufer variables for perturbations of Jacobi and CMV matrices,
 \emph{J.\ Math.\ Anal.\ Appl.,}\ in press. DOI:10.1016/j.jmaa.2016.07.036. (arXiv:1409.7116).

\bibitem{MO} 
 P.\ Munger and D.\ Ong,
 {The H\"older continuity of spectral measures of an extended CMV matrix},
 \emph{J.\ Math.\ Phys.,}\ \textbf{55} (2014), 093507, 10 pp.

\bibitem{O14} 
 D.\ Ong,
 {Purely singular continuous spectrum for CMV operators generated by subshifts},
 \emph{J.\ Stat.\ Phys.,}\ \textbf{155} (2014), 763--776.

\bibitem{reedsimon1} 
 M.\ Reed and B.\ Simon,
 \emph{Methods of Modern Mathematical Physics, I: Functional Analysis},
 Academic Press, New York, 1972.

\bibitem{sacksell1} 
 R.\ Sacker and G.\ Sell,
 {Existence of dichotomies and invariant splittings for linear differential systems I.},
 \emph{J.\ Diff.\ Eq.,}\ \textbf{15} (1974), 429--458.

\bibitem{sacksell2} 
 R.\ Sacker and G.\ Sell,
 {A spectral theory for linear differential systems},
 \emph{J.\ Diff.\ Eq.},\ \textbf{27} (1978), 320--358.

\bibitem{selgrade} 
 J.\ Selgrade,
 {Isolated invariant sets for flows on vector bundles},
 \emph{Trans.\ Amer.\ Math.\ Soc.,}\ \textbf{203} (1975), 359--390.

\bibitem{S05} 
 B.\ Simon,
 \emph{Orthogonal Polynomials on the Unit Circle. Part 1. Classical Theory},
 American Mathematical Society Colloquium Publications \textbf{54}, Part 1, American Mathematical Society, Providence, RI, 2005.

\bibitem{Y04} 
 J.-C.\ Yoccoz,
 Some questions and remarks about $\mathrm{SL}(2,\R)$ cocycles,
 \emph{Modern Dynamical Systems and Applications}, 447--458, Cambridge Univ.\ Press, Cambridge, 2004.

\bibitem{Z}
 Z.\ Zhang,
 Resolvent set of Schr\"odinger operators and uniform hyperbolicity,
 preprint, (arXiv:1305.4226).

\end{thebibliography}
\end{document}